\documentclass[11pt]{article}

\usepackage[margin=1in]{geometry}
\usepackage{amsmath,amssymb,amsthm,mathtools}
\usepackage{microtype}
\usepackage{enumitem}
\usepackage[hidelinks]{hyperref}

\allowdisplaybreaks
\setlist[enumerate]{leftmargin=*,itemsep=2pt,topsep=4pt}
\setlist[itemize]{leftmargin=*,itemsep=2pt,topsep=4pt}

\newtheorem{theorem}{Theorem}[section]
\newtheorem{proposition}[theorem]{Proposition}
\newtheorem{lemma}[theorem]{Lemma}
\newtheorem{corollary}[theorem]{Corollary}
\theoremstyle{definition}

\theoremstyle{remark}

\newcommand{\cF}{\mathcal F}

\newcommand{\cE}{\mathcal E}
\newcommand{\cU}{\mathcal U}
\newcommand{\Om}{\Omega}
\newcommand{\Def}{\operatorname{Def}}
\newcommand{\dotcupunion}{\mathbin{\dot\cup}}
\newcommand{\clin}{c^{\mathrm{lin}}}

\title{Robust Repulsion for Growing Crowns\\
in Linear Hypergraphs}
\author{Mahesh Ramani}
\date{16 August 2026}

\begin{document}
\maketitle

\begin{abstract}
Put \(q=r-1\), \(t=k-1\), and \(D=tq+1\).  For an edge \(e\) of a
linear \(C^r_{1,k}\)-free \(r\)-uniform hypergraph, define
\[
 \delta_H(e)=\sum_{v\in e}\frac1{d_H(v)}-\frac rD.
\]
The defect satisfies \(\delta_H(e)\ge0\).  At equality, every vertex of
\(e\) has degree \(D\), and the petal trace at \(e\) is a disjoint
union of \(t\) affine planes of order \(q\).  Equivalently, restoring the
base line gives \(t\) projective planes of order \(q\) with common line
\(e\).

For growing crowns, the equality structure is stable in the following sense.  If
\(q_j\to\infty\), \(2\le t_j\le q_j\), and \(e_j\) is an edge of a finite
linear \(C_{1,t_j+1}^{q_j+1}\)-free hypergraph satisfying
\[
 \frac{t_j^2}{q_j}\to0,
 \qquad
 t_j^3\delta_{H_j}(e_j)\to0,
\]
then, for every fixed \(0<\theta<1\),
\[
 \frac{
 |\{f\ne e_j:f\cap e_j\ne\varnothing,\
 \delta_{H_j}(f)\ge\theta/t_j^2\}|
 }{(q_j+1)(t_j q_j)}
 \to1.
\]
Thus an edge close to equality is adjacent almost entirely to edges with
defect of order at least \(t_j^{-2}\).  A uniform form gives absolute
constants \(Q_0,\varepsilon_0,c_0>0\) such that, whenever
\(q\ge Q_0t^2\) and \(\delta_H(e)<\varepsilon_0/t^3\), at least
\(\tfrac12(q+1)tq\) neighbors of \(e\) have defect at least
\(1/(100t^2)\).  Consequently,
\[
 \clin_{q+1,t+1}\le t-\frac{c_0}{t},
\]
where \(\clin_{r,k}\) denotes the asymptotic linear Tur\'an coefficient
for \(C^r_{1,k}\).
\end{abstract}

\section{Introduction}\label{sec:intro}

For \(3\le k\le r\), the \(r\)-uniform \(k\)-crown
\(C^r_{1,k}\) consists of a base edge \(e_0\) and pairwise disjoint
petals \(e_1,\ldots,e_k\) satisfying
\[
 e_0\cap e_i=\{v_i\}\qquad(i\in[k]),
\]
where \(v_1,\ldots,v_k\) are distinct points of \(e_0\).  Generalized
crowns in linear hypergraphs were studied by Zhang, Broersma, and Wang
\cite{ZhangBroersmaWang2025}; Adak obtained the reciprocal-degree
localization used below \cite{Adak2026}.

Set
\begin{equation}\label{eq:parameters}
 q=r-1,\qquad t=k-1,\qquad
 D=tq+1,\qquad D^-=(t-1)q+1.
\end{equation}
For an edge \(e\), write
\[
 \Phi_H(e)=\sum_{v\in e}\frac1{d_H(v)},
 \qquad
 \delta_H(e)=\Phi_H(e)-\frac rD.
\]
The local bound is \(\delta_H(e)\ge0\).  Equality forces every degree
on \(e\) to equal \(D\), and the petal trace decomposes into \(t\) affine
planes of order \(q\).  A quantitative form of the localization argument is
needed for stability; standard facts on affine and projective planes may be
found in \cite{Dembowski1968}.

The stability problem is to control a second edge meeting an edge for which
\(\delta_H(e)\) is small.  The useful error parameter is the number of
uncovered point pairs in large subsets of the affine-plane structure forced
by equality.  Linearity makes these uncovered pairs a resource shared by
external edges: two distinct external edges cannot use the same pair.  A
second low-defect edge, on the other hand, forces many external
intersections.  Comparing these upper and lower bounds yields the local
repulsion theorem and, after double counting, a uniform gap in the
asymptotic linear Tur\'an coefficient.

\section{Colored traces and the equality configuration}\label{sec:traces}

Fix an edge
\[
                         e=\{v_1,\ldots,v_r\}
\]
of a linear \(r\)-uniform hypergraph \(H\).  For \(i\in[r]\), define
the \(i\)-th petal family
\[
\cF_i(e)=
\{f\setminus\{v_i\}:f\in E(H)\setminus\{e\},\ v_i\in f\}.
\]
Every member is a \(q\)-set outside \(e\), and
\[
                 x_i:=|\cF_i(e)|=d_H(v_i)-1.
\]
The index \(i\) is regarded as a color.

\begin{lemma}\label{lem:trace-linear}
Each \(\cF_i(e)\) is a matching.  If \(i\ne j\),
\(A\in\cF_i(e)\), and \(B\in\cF_j(e)\), then
\(|A\cap B|\le1\).
\end{lemma}

\begin{proof}
Two distinct hyperedges through \(v_i\) already meet at \(v_i\), so
linearity makes their remaining \(q\)-sets disjoint.  Hyperedges
corresponding to different colors meet \(e\) at different points, and
linearity permits at most one further intersection.
\end{proof}

A \emph{rainbow matching} is a family of pairwise disjoint petals of
distinct colors.

\begin{proposition}[Crown--rainbow equivalence]\label{prop:crown-rainbow}
The edge \(e\) is the base of a copy of \(C_{1,k}^r\) if and only if
\(\cF_1(e),\ldots,\cF_r(e)\) contain a rainbow matching of size \(k\).
\end{proposition}

\begin{proof}
Delete from each crown petal its point on \(e\).  This gives a rainbow
matching.  Conversely, adjoin the appropriate base point to every
member of a rainbow matching.
\end{proof}

Let
\[
                         \mathcal M=\{A_1,\ldots,A_h\}
\]
be a maximum rainbow matching, where \(A_j\) has color \(c_j\), and put
\[
                         U_{\mathcal M}=A_1\cup\cdots\cup A_h.
\]
Every petal of a color not used by \(\mathcal M\) meets
\(U_{\mathcal M}\), or the matching could be enlarged.  For such a
petal \(A\), define
\[
T_{\mathcal M}(A)=
\{j\in[h]:A\cap A_j\ne\varnothing\}
\]
and
\[
J_{\mathcal M}=
\sum_{\substack{A\text{ of a color}\\\text{unused by }\mathcal M}}
\bigl(|T_{\mathcal M}(A)|-1\bigr)\ge0.
\]

\begin{proposition}[Matching-blocker inequality]\label{prop:blocker}
Let \(C(\mathcal M)=\{c_1,\ldots,c_h\}\).  Then
\[
\sum_{i\notin C(\mathcal M)}x_i+J_{\mathcal M}
\le
\sum_{w\in U_{\mathcal M}}
\min\{r-h,d_H(w)-1\}.
\]
\end{proposition}

\begin{proof}
The left side is the total number of incidences between unused-color
petals and \(U_{\mathcal M}\).  Indeed, the selected petals are
disjoint and a petal meets each of them in at most one point.
At a fixed \(w\in U_{\mathcal M}\), at most one petal from each unused
color contains \(w\), giving at most \(r-h\) incidences.  Exactly one
selected hyperedge contains \(w\), so at most \(d_H(w)-1\) other
hyperedges can contribute.  Summing the smaller bound proves the proposition.
\end{proof}

\subsection{The equality case for colored traces}

Consider an abstract collection
\[
                         \cF_1,\ldots,\cF_r
\]
of matchings of \(q\)-sets.  Sets of different colors meet in at most
one point.  Assume
\[
                         |\cF_i|=tq\qquad(i\in[r])
\]
and that there is no rainbow matching of size \(t+1\).

\begin{lemma}[Blocking at equality]\label{lem:perfect-blocker}
Let \(\mathcal M\) be a rainbow matching of size \(t\).  For every
unused color, its \(tq\) sets meet \(U_{\mathcal M}\) in singletons,
and those singletons partition \(U_{\mathcal M}\).
\end{lemma}

\begin{proof}
Every unused-color set meets \(U_{\mathcal M}\), or it augments
\(\mathcal M\).  The \(tq\) sets of one color are disjoint, and their
nonempty intersections with the \(tq\)-point set \(U_{\mathcal M}\)
are disjoint.  Hence all intersections are singletons and cover the
union.
\end{proof}

\begin{lemma}[Extendability]\label{lem:extend}
Every rainbow matching of size at most \(t\) extends to one of size
\(t\).  If a color is unused by the initial matching, the extension can
be chosen to avoid that color.
\end{lemma}

\begin{proof}
Suppose the current matching has size \(h<t\), and let \(j\) be a color
that is required to remain unused.  At most \(h+1\le t\) colors are
used or forbidden.  Since \(r=q+1\ge t+1\), some color \(i\) is neither.
Each selected \(q\)-set meets at most \(q\) members of the matching
\(\cF_i\), so the \(h\) selected sets block at most
\(hq<tq=|\cF_i|\) members of \(\cF_i\).  Add an unblocked member and
repeat.
\end{proof}

Join two sets of different colors when they intersect.

\begin{lemma}[Exact pair degree]\label{lem:pair-degree}
Every set meets exactly \(q\) sets of every other color.
\end{lemma}

\begin{proof}
A \(q\)-set meets at most \(q\) members of a matching.  Fix
\(A\in\cF_i\) and another color \(j\).  Extend \(\{A\}\) to a rainbow
\(t\)-matching avoiding color \(j\), using Lemma \ref{lem:extend}.
By Lemma \ref{lem:perfect-blocker}, the color-\(j\) sets partition the
matching union by singleton intersections.  Exactly one passes through
each point of \(A\).
\end{proof}

\begin{lemma}[Transitivity of intersections]\label{lem:transitivity}
Let \(A,B,C\) have three distinct colors.  If \(A\) meets both \(B\)
and \(C\), then \(B\) meets \(C\).
\end{lemma}

\begin{proof}
If \(B\) and \(C\) were disjoint, extend \(\{B,C\}\) to a rainbow
\(t\)-matching avoiding the color of \(A\).  The set \(A\) would meet
the matching union at least twice, contrary to
Lemma \ref{lem:perfect-blocker}.
\end{proof}

\begin{lemma}[Pairwise block decomposition]\label{lem:block-decomp}
For every pair of colors, the bipartite intersection graph between them is
a disjoint union of \(t\) copies of \(K_{q,q}\).
\end{lemma}

\begin{proof}
By Lemma \ref{lem:pair-degree}, the intersection graph is \(q\)-regular
with \(tq\) vertices in each part.  Suppose \(A,A'\) in the first part
have a common neighbor \(B\).  Choose a third color and a set \(C\) of
that color meeting \(A\).  For any neighbor \(B'\) of \(A\), repeated
applications of Lemma \ref{lem:transitivity} give
\[
A\sim B,\ A\sim C\Rightarrow B\sim C,\qquad
A\sim B',\ A\sim C\Rightarrow B'\sim C,
\]
\[
B\sim A',\ B\sim C\Rightarrow A'\sim C,\qquad
C\sim A',\ C\sim B'\Rightarrow A'\sim B'.
\]
Thus \(N(A)\subseteq N(A')\), and equality follows from regularity.
Along every even path, all vertices in the same part have the same
neighborhood.  Every connected component is therefore complete
bipartite.  Regularity makes it \(K_{q,q}\), and there are \(t\)
components.
\end{proof}

\begin{lemma}[Common labeling]\label{lem:common-label}
There are partitions
\[
       \cF_i=\cF_i^1\dotcupunion\cdots\dotcupunion\cF_i^t,
       \qquad |\cF_i^b|=q,
\]
such that two sets of different colors intersect if and only if they
have the same label.
\end{lemma}

\begin{proof}
Fix color \(1\).  For each \(j\ne1\), Lemma \ref{lem:block-decomp}
partitions \(\cF_1\) into common-neighborhood classes of size \(q\).
It remains to verify that this partition is independent of \(j\).

Suppose \(A,A'\in\cF_1\) have a common color-\(j\) neighbor \(B\).
Let \(\ell\notin\{1,j\}\) and choose \(C\in\cF_\ell\) meeting \(A\).
Transitivity first gives \(B\sim C\), and then \(A'\sim C\).  Hence
every color-\(\ell\) neighbor of \(A\) is a neighbor of \(A'\).
Both have \(q\) such neighbors, so the neighborhoods are equal.  The
same is true for every nonreference color.  The partitions of
\(\cF_1\) therefore coincide.

Label the common classes \(1,\ldots,t\).  A set of another color
receives the label of its neighborhood in \(\cF_1\).  Every label class
has size \(q\).  If two nonreference sets have the same label, a member
of the corresponding reference class meets both, so transitivity makes
them intersect.  Conversely, if they intersect, transitivity with a
reference neighbor shows that their reference neighborhoods coincide.
\end{proof}

\begin{theorem}[Colored affine-plane classification]\label{thm:colored-plane}
Under the hypotheses of this subsection, the colored system is the disjoint union of \(t\) affine
planes of order \(q\).  More precisely, for every \(b\in[t]\) there is
a \(q^2\)-point set \(\Om_b\) such that
\[
                         \bigcup_{i=1}^r\cF_i^b
\]
is the block set of an affine plane of order \(q\), with the
\(\cF_i^b\) as its \(q+1\) parallel classes.  The sets
\(\Om_1,\ldots,\Om_t\) are pairwise disjoint.
\end{theorem}

\begin{proof}
Fix \(b\) and a color \(i\).  The \(q\) members of \(\cF_i^b\) are
disjoint \(q\)-sets, so their union \(\Om_b\) has size \(q^2\).
If \(j\ne i\) and \(B\in\cF_j^b\), then \(B\) meets every member of
\(\cF_i^b\).  These \(q\) intersections are distinct and exhaust
\(B\), so \(B\subseteq\Om_b\).  Thus every \(\cF_j^b\) is a parallel
class partitioning the same point set.

There are \(q(q+1)\) blocks.  No pair of points occurs twice, and the
number of covered pairs is
\[
                         q(q+1)\binom q2=\binom{q^2}2.
\]
Every pair is therefore covered once.  The blocks form an
\(S(2,q,q^2)\), equivalently an affine plane of order \(q\).
Different labels are disjoint by Lemma \ref{lem:common-label}.
\end{proof}

\section{Reciprocal localization and equality}\label{sec:localization}

The following reciprocal-degree localization is used in the stability
argument and is included to keep the crown analysis self-contained.

\begin{lemma}[Greedy crown lemma]\label{lem:greedy-crown}
Let \(G\) be a linear \(r\)-uniform hypergraph.  Suppose that an edge
\(e\) contains distinct vertices \(v_1,\ldots,v_k\) satisfying
\[
        d_G(v_i)\ge(i-1)(r-1)+2\qquad(i\in[k]).
\]
Then \(e\) is the base of a copy of \(C_{1,k}^r\).
\end{lemma}

\begin{proof}
Choose the petals in the order \(v_1,\ldots,v_k\).  Suppose that
pairwise disjoint petals \(f_1,\ldots,f_{i-1}\) have been chosen, with
\(f_j\cap e=\{v_j\}\).  There are at least
\((i-1)(r-1)+1\) candidate edges other than \(e\) through \(v_i\).
For fixed \(j<i\), the edge \(f_j\) blocks at most \(r-1\) candidates:
each candidate meeting \(f_j\) uses one of the \(r-1\) points of
\(f_j\setminus\{v_j\}\), and linearity permits at most one candidate
through \(v_i\) and any prescribed such point.  The previously chosen
petals block at most \((i-1)(r-1)\) candidates.  A candidate remains,
and induction gives the required crown.
\end{proof}

\begin{lemma}[Degree witness]\label{lem:degree-witness}
Let \(H\) be linear and \(C_{1,k}^r\)-free.  If
\(e=\{u_1,\ldots,u_r\}\) is ordered so that
\(d_H(u_1)\ge\cdots\ge d_H(u_r)\), then some \(i\in[k]\) satisfies
\[
              d_H(u_i)\le(k-i)(r-1)+1.
\]
\end{lemma}

\begin{proof}
Otherwise set \(v_j=u_{k-j+1}\) for \(j\in[k]\).  The reversed vertices satisfy the degree condition in Lemma~\ref{lem:greedy-crown}, a contradiction.
\end{proof}

\begin{proposition}[Edgewise localization]\label{prop:localization}
Every edge \(e\) of a linear \(C_{1,k}^r\)-free \(r\)-graph satisfies
\[
                         \Phi_H(e)\ge\frac rD.
\]
Equality holds if and only if every vertex of \(e\) has degree \(D\).
\end{proposition}

\begin{proof}
Order \(e=\{u_1,\ldots,u_r\}\) by decreasing degree, take the witness
\(i\) in Lemma \ref{lem:degree-witness}, and put
\(B_i=(k-i)(r-1)+1\).  Then \(d_H(u_j)\le B_i\) for
\(j\ge i\), whence
\[
             \Phi_H(e)\ge\frac{r-i+1}{B_i}.
\]
A direct subtraction gives
\[
 \frac{r-i+1}{B_i}-\frac rD
 =\frac{(i-1)\{r(r-k)+k-2\}}{B_i D}.
\]
The numerator in braces is nonnegative because
\(r(r-k)+k-2\ge r-2>0\).  This proves the required inequality.

Equality forces \(i=1\) in the displayed gap identity.  Then
\(d_H(u_1)\le D\), hence every degree on \(e\) is at most \(D\), and
\(\sum_{u\in e}1/d_H(u)=r/D\) forces all of them to equal \(D\).
The converse is immediate.
\end{proof}

For an edge \(e\), recall
\[
 \delta_H(e)=\Phi_H(e)-\frac rD,
 \qquad
 L(e)=\sum_{v\in e}(D-d_H(v)).
\]
The summands defining \(L(e)\) will only be used after nonnegativity
has been established.  Put
\[
 A_{q,t}=q^2-(t-1)q-1,
 \qquad
 G_{q,t}=\frac{A_{q,t}}{DD^-}.
\]
Since \(q\ge t\), one has \(A_{q,t}>0\).

\begin{proposition}[Degree stability]\label{prop:degree-stability}
If \(\delta_H(e)<G_{q,t}\), then every vertex of \(e\) has degree at
most \(D\), and
\[
                         L(e)\le D^2\delta_H(e).
\]
\end{proposition}

\begin{proof}
For a witness of index \(i\ge2\), the displayed gap identity, with
\(r=q+1\) and \(k=t+1\), gives
\[
 \delta_H(e)\ge
 \frac{(i-1)A_{q,t}}{((t+1-i)q+1)D}.
\]
The ratio \((i-1)/((t+1-i)q+1)\) is strictly increasing, so its
minimum for \(i\ge2\) occurs at \(i=2\) and equals \(1/D^-\).
Thus \(\delta_H(e)<G_{q,t}\) forces a witness of index one.  As in the
equality proof, all degrees on \(e\) are then at most \(D\).
For \(d\le D\),
\[
                   \frac1d-\frac1D
                   =\frac{D-d}{Dd}\ge\frac{D-d}{D^2}.
\]
Summing over \(e\) proves the claim.
\end{proof}

\begin{theorem}[Equality structure]\label{thm:sharp-neighborhood}
Let \(e\) be an edge of a linear \(C_{1,t+1}^{q+1}\)-free hypergraph.
The following are equivalent:
\begin{enumerate}
\item \(\delta_H(e)=0\);
\item every point of \(e\) has degree \(D\);
\item the petal trace at \(e\) is a disjoint union of \(t\) affine
      planes of order \(q\), with the colors as parallel classes.
\end{enumerate}
For each affine packet, adjoining the base point \(v_i\) to every line
of color \(i\), and adjoining the base edge \(e\), produces a projective
plane of order \(q\).  The resulting \(t\) projective planes have common
line \(e\) and are otherwise disjoint.
\end{theorem}

\begin{proof}
The equivalence of the first two statements is
Proposition~\ref{prop:localization}.  If every point of \(e\) has
degree \(D\), then every color has \(D-1=tq\) petals.  By
Proposition~\ref{prop:crown-rainbow}, crown-freeness excludes a rainbow
matching of size \(t+1\), and Theorem~\ref{thm:colored-plane} gives the
affine-plane decomposition.  Conversely, the decomposition supplies
\(tq\) petals of every color, so every point of \(e\) has degree \(D\).

Fix one affine packet.  Add the \(q+1\) points of \(e\) as points at
infinity, assign \(v_i\) to the parallel class of color \(i\), and take
\(e\) as the line at infinity.  Lines in one affine class meet at their
assigned point, lines in distinct classes retain their unique affine
intersection, and every affine point lies on one line of each class.
This is a projective plane of order \(q\).  The affine packets are
disjoint, so the projective planes intersect exactly in \(e\).
\end{proof}
\section{Approximate blocking near equality}\label{sec:approx}

Assume throughout this section that \(e=\{v_1,\ldots,v_r\}\) satisfies
the hypothesis of Proposition \ref{prop:degree-stability}.  Write
\[
 \Delta_i=D-d_H(v_i),
 \qquad
 x_i=|\cF_i(e)|=tq-\Delta_i,
 \qquad
 \sum_{i=1}^r\Delta_i=L(e).
\]
Thus all \(\Delta_i\) are nonnegative integers.

\begin{proposition}[Approximate blocking]
\label{prop:approx-blocking}
Suppose
\begin{equation}\label{eq:approx-threshold}
 L(e)<q(q-t+2).
\end{equation}
Then the trace at \(e\) has a maximum rainbow matching
\(\mathcal M=\{A_1,\ldots,A_t\}\) of size exactly \(t\).  Put
\(U=A_1\cup\cdots\cup A_t\).  For every color \(i\) unused by
\(\mathcal M\):
\begin{enumerate}
\item at most \(\Delta_i\) members of \(\cF_i(e)\) meet at least two
      branches \(A_b\);
\item at most \(\Delta_i\) points of \(U\) are uncovered by color
      \(i\);
\item for every \(b\in[t]\), at least \(q-2\Delta_i\) color-\(i\)
      rows have trace on \(U\) equal to a singleton in \(A_b\);
\item singleton rows of two distinct unused colors with the same branch
      label intersect;
\item if \(B\) is a singleton row of label \(b\) and \(j\ne i\) is
      another unused color, then \(B\) meets at most \(2\Delta_j\)
      color-\(j\) rows not having singleton label \(b\).
\end{enumerate}
\end{proposition}

\begin{proof}
Let \(h\le t\) be the size of a maximum rainbow matching and let
\(U_{\mathcal M}\) be its union.  The unused colors contain at least
\[
             (r-h)tq-L(e)
\]
rows.  Proposition \ref{prop:blocker}, with
\(|U_{\mathcal M}|=hq\), bounds this number above by \(hq(r-h)\).
If \(h<t\), then
\[
                    L(e)\ge(t-h)q(r-h)
                          \ge q(q-t+2),
\]
contrary to \eqref{eq:approx-threshold}.  Hence \(h=t\).

Fix an unused color \(i\).  Maximality makes every one of its
\(tq-\Delta_i\) rows meet \(U\).  Since the color is a matching, its
row traces on the \(tq\)-point set \(U\) are disjoint.  Therefore
\[
 \sum_{B\in\cF_i(e)}(|B\cap U|-1)\le\Delta_i.
\]
This proves part 1.  The nonempty disjoint traces cover at least
\(tq-\Delta_i\) points, proving part 2.  On a fixed branch \(A_b\), at
most \(\Delta_i\) points are uncovered and at most \(\Delta_i\) more
belong to nonsingleton traces.  Every remaining point lies in a
distinct singleton row of label \(b\), proving part 3.

For part 4, suppose singleton rows \(B,C\) of distinct unused colors
and common label \(b\) were disjoint.  Then
\[
              \{B,C\}\cup\{A_a:a\ne b\}
\]
would be a rainbow matching of size \(t+1\), a contradiction.  Finally,
part 3 supplies at least \(q-2\Delta_j\) color-\(j\) singleton rows of
label \(b\).  Part 4 makes all of them meet \(B\), at distinct points
because color \(j\) is a matching.  At most \(2\Delta_j\) points of
\(B\) remain available to rows of color \(j\) with another trace type.
This proves part 5.
\end{proof}

\section{Uncovered pairs}\label{sec:pairs}

Let \(\mathcal R\) be a linear family of sets, let \(\Om\) be a point
set, and write \(s_B=|B\cap\Om|\).  Let \(M_\Om\) be the incidence
matrix of the restricted rows \(B\cap\Om\), and let \(d_\Om(x)\) be
the corresponding column degree.  Define
\[
 \cU(\Om)=\binom{|\Om|}{2}-\sum_{B\in\mathcal R}\binom{s_B}{2}.
\]

\begin{proposition}[Pair-defect identity]\label{prop:pair-defect}
One has \(\cU(\Om)\ge0\) and
\[
\boxed{
 \cU(\Om)=\binom{|\Om|}{2}
 -\frac12\left\|M_\Om^{\mathsf T}M_\Om-
 \operatorname{diag}(d_\Om(x):x\in\Om)\right\|_F^2.}
\]
Suppose that the rows in \(\mathcal R\) arise from hyperedges in a
linear hypergraph.  If every hyperedge \(g\) in a family \(\mathcal E\)
is distinct from those row-edges, then
\[
 \#\{g\in\mathcal E:|g\cap\Om|\ge h\}
 \le\frac{\cU(\Om)}{\binom h2}\qquad(h\ge2),
\]
provided every pair of \(g\cap\Om\) is required to be uncovered by
\(\mathcal R\).  This requirement holds, in particular, when the
row-edges all meet a fixed edge \(e\) and the members of \(\mathcal E\)
are disjoint from \(e\).
\end{proposition}

\begin{proof}
Linearity makes the two-point subsets contributed by distinct
restricted rows disjoint, proving nonnegativity.  Every off-diagonal
entry of \(M_\Om^{\mathsf T}M_\Om\) is zero or one, and it equals one
precisely on a covered ordered pair.  Its diagonal is \(d_\Om(x)\), so
the squared Frobenius norm in the displayed identity is twice the number of covered
unordered pairs.

If all pairs in \(g\cap\Om\) are uncovered, then \(g\) consumes at
least \(\binom h2\) uncovered pairs.  Distinct hyperedges consume
disjoint pairs, since otherwise they share two points.  This proves the claimed bound.  In the stated special case, a pair contained both in \(g\) and
in a row-edge would make those two hyperedges intersect twice.
\end{proof}

\section{Packet construction and uncovered pairs}\label{sec:packets}

Assume the hypotheses of Proposition~\ref{prop:approx-blocking}, and
fix a maximum rainbow matching
\(\mathcal M=\{A_1,\ldots,A_t\}\).  Let \(J\) be the set of unused
colors, so
\[
 p_0:=|J|=q+1-t.
\]
For \(i\in J\) and \(b\in[t]\), let \(\cF_i^b\) be the rows of color
\(i\) whose trace on \(A_1\cup\cdots\cup A_t\) is a singleton in
\(A_b\).  Put
\[
 \xi=\frac{L(e)}{q^2}.
\]

\begin{theorem}[Packet construction and uncovered pairs]
\label{thm:least-deficient}
Suppose \(p_0\ge2\).  Choose \(p,c\in J\) with the two smallest
values of \(\Delta_i\), and define
\[
 w=2(\Delta_p+\Delta_c),\qquad
 \rho=\frac wq.
\]
Then
\begin{equation}\label{eq:rho-bound}
 \rho\le\frac{4L(e)}{q(q+1-t)}.
\end{equation}
Assume \(\rho<1\), and for \(b\in[t]\) set
\[
 \Om_b=\{R\cap C:R\in\cF_p^b,\ C\in\cF_c^b\}.
\]
The sets \(\Om_1,\ldots,\Om_t\) are pairwise disjoint and satisfy:
\begin{enumerate}
\item
\[
 |\Om_b|\ge(1-\rho)q^2;
\]
\item every unused-color singleton row of label \(b\) contains at least
      \((1-\rho)q\) points of \(\Om_b\);
\item if
      \[
      \mathcal R_b=\bigcup_{i\in J}\cF_i^b,
      \]
      then
      \[
      |\mathcal R_b|\ge q^2-(t-1)q-2L(e);
      \]
\item with \(s_B=|B\cap\Om_b|\), define
      \[
      \cU_b=\binom{|\Om_b|}{2}
      -\sum_{B\in\mathcal R_b}\binom{s_B}{2}.
      \]
      Then
      \[
      0\le\cU_b\le\frac{\eta_*q^4}{2},
      \qquad
      \eta_*=\frac tq+2\xi+2\rho;
      \]
\item let \(\mathcal E\) be any family of hyperedges disjoint from
      \(e\), and put \(h_b(g)=|g\cap\Om_b|\).  If \(\beta q>t\), then
      \[
      \left|\left\{g\in\mathcal E:
      \sum_{b=1}^t h_b(g)\ge\beta q\right\}\right|
      \le
      \frac{t^2\eta_*}{\beta(\beta-t/q)}\,q^2.
      \]
\end{enumerate}
Moreover,
\begin{equation}\label{eq:eta-star-bound}
 \eta_*
 \le \frac tq+2\xi+
 \frac{8\xi}{1-(t-1)/q}.
\end{equation}
\end{theorem}

\begin{proof}
The average of the two smallest nonnegative values is no larger than
the average of all values, so
\[
 \Delta_p+\Delta_c
 \le\frac{2}{p_0}\sum_{i\in J}\Delta_i
 \le\frac{2L(e)}{q+1-t}.
\]
This proves \eqref{eq:rho-bound}.  Proposition~\ref{prop:approx-blocking}
gives
\begin{equation}\label{eq:singleton-count}
 |\cF_i^b|\ge q-2\Delta_i
 \qquad(i\in J,\ b\in[t]).
\end{equation}

Every \(p\)-row and \(c\)-row with the same label intersects, and the
intersection map is injective.  Hence
\[
 |\Om_b|
 =|\cF_p^b|\,|\cF_c^b|
 \ge(q-2\Delta_p)(q-2\Delta_c)
 \ge q^2-wq.
\]
Packet disjointness follows from the color-\(p\) matching.

Let \(B\in\cF_i^b\).  If \(i\notin\{p,c\}\), then \(B\) meets every
retained \(p\)-row and every retained \(c\)-row.  Inclusion--exclusion
inside the \(q\)-set \(B\) gives
\[
 |B\cap\Om_b|
 \ge|\cF_p^b|+|\cF_c^b|-q
 \ge q-w.
\]
The cases \(i=p,c\) follow by intersecting with the opposite defining
class.  This proves part 2.  Summing
\eqref{eq:singleton-count} over \(i\in J\) gives
\[
 |\mathcal R_b|
 \ge(q+1-t)q-2\sum_{i\in J}\Delta_i
 \ge q^2-(t-1)q-2L(e),
\]
which is part 3.

Part 4 uses Proposition~\ref{prop:pair-defect}.  Since
\(|\Om_b|\le q^2\) and every row in \(\mathcal R_b\) has at least
\(q-w\) points in \(\Om_b\),
\[
 \cU_b
 \le\binom{q^2}{2}
 -\bigl(q^2-(t-1)q-2L(e)\bigr)\binom{q-w}{2}.
\]
Put
\[
 x=\frac{t-1}{q}+2\xi,\qquad
 y=1-\left((1-\rho)^2-\frac1q\right).
\]
If \(\eta_*\ge1\), the claimed upper bound follows from
\(\cU_b\le\binom{q^2}{2}\).  Otherwise \(x,y\in[0,1]\),
\(y\le2\rho+1/q\), and \(x+y\le\eta_*\).  The two factors subtracted
above are at least \(q^2(1-x)\) and \(q^2(1-y)/2\), respectively.
Since \((1-x)(1-y)\ge1-x-y\), part 4 follows.

For part 5, an external edge \(g\) consumes
\[
 \sum_{b=1}^t\binom{h_b(g)}{2}
\]
uncovered within-packet pairs, and distinct external edges consume
disjoint pairs.  If
\[
 H_g=\sum_{b=1}^t h_b(g)\ge\beta q,
\]
then Cauchy--Schwarz gives
\[
 \sum_{b=1}^t\binom{h_b(g)}{2}
 =\frac12\left(\sum_b h_b(g)^2-H_g\right)
 \ge\frac12\left(\frac{H_g^2}{t}-H_g\right)
 \ge\frac{\beta q(\beta q-t)}{2t}.
\]
Since
\[
 \sum_{b=1}^t\cU_b\le\frac{t\eta_*q^4}{2},
\]
division gives the stated bound.

Finally, \eqref{eq:rho-bound} and \(L(e)=\xi q^2\) give
\[
 2\rho\le\frac{8\xi}{1-(t-1)/q},
\]
which proves \eqref{eq:eta-star-bound}.
\end{proof}
\section{Prescribed-petal intersections}\label{sec:capacity}

The following bound links the reciprocal defect at one edge to the
uncovered-pair estimate associated with another.  The first petal in the
resulting rainbow matching is prescribed.

\begin{theorem}[Prescribed-petal intersection bound]\label{thm:capacity}
Let \(f\) be an edge of a linear
\(C_{1,t+1}^{q+1}\)-free hypergraph, suppose
\(d_H(z)\le D\) for every \(z\in f\), and put
\[
 \Delta_z=D-d_H(z),\qquad L(f)=\sum_{z\in f}\Delta_z.
\]
Fix \(x\in f\), and assume
\[
                         L(f)<q(q-t+1).
\]
If \(g\ne f\) meets \(f\) at a point different from \(x\), then \(g\)
meets at least \(q-\Delta_x\) distinct edges in
\[
                    \cE_x(f)=\{h\ne f:x\in h\}.
\]
The corresponding intersection points on \(g\) are distinct.
\end{theorem}

\begin{proof}
Suppose \(g\cap f=\{y\}\), where \(y\ne x\).  Seed a rainbow matching
in the petal trace at \(f\) with \(g\setminus\{y\}\), and extend it
maximally using only the \(q\) colors in \(f\setminus\{x\}\).  Let its
size be \(h\le t\), and let \(U\) be the union of its \(h\) disjoint
\(q\)-set traces.

If \(h<t\), there are \(q-h\) unused permitted colors.  They contain at
least
\[
                         (q-h)tq-L(f)
\]
petals.  Every such petal meets \(U\), by maximality.  At each point of
\(U\), at most one petal of each unused color occurs.  Therefore
\[
             (q-h)tq-L(f)\le |U|(q-h)=hq(q-h),
\]
and hence
\[
                         L(f)\ge q(q-h)(t-h).
\]
For \(h\le t-1\), the right side is at least
\(q(q-t+1)\): the function \((q-h)(t-h)\) decreases with \(h\) on this
range and has value \(q-t+1\) at \(h=t-1\).  This contradicts the assumed bound on \(L(f)\), so
\(h=t\).

Write the selected actual hyperedges as
\(g=g_1,g_2,\ldots,g_t\).  There are
\[
                         |\cE_x(f)|=d_H(x)-1=tq-\Delta_x
\]
edges in the opposite star.  Every one meets at least one \(g_j\), for
otherwise that edge together with \(g_1,\ldots,g_t\) would be a
\((t+1)\)-crown with base \(f\).  Each \(g_j\) meets at most \(q\)
members of \(\cE_x(f)\): its point on \(f\) cannot lie on an
\(x\)-edge, and each of its other \(q\) points lies on at most one
\(x\)-edge by linearity.  The last \(t-1\) selected edges therefore
meet at most \((t-1)q\) members of the opposite star.  The prescribed
edge \(g_1=g\) meets at least
\[
                    tq-\Delta_x-(t-1)q=q-\Delta_x
\]
of them.  Two such opposite-star edges cannot meet \(g\) at the same
point, since they already share \(x\).
\end{proof}

\section{Comparing external demand with uncovered pairs}\label{sec:criterion}

Fix \(0<a<1\).  An unused color \(i\) is \emph{good} if
\(\Delta_i\le aq\).  For a good color \(i\), let \(X_i\) contain every
point of every nonsingleton color-\(i\) trace row and, from each
singleton row of label \(b\), every point outside \(\Om_b\).
A neighbor \(f\) of \(e\) is \emph{retained} if its color is unused
and good and its trace row has a singleton branch label.

\begin{lemma}[Spill bound]\label{lem:spill}
For every good unused color \(i\),
\[
 |X_i|\le(a+t\rho)q^2.
\]
Let \(f\) be a retained neighbor of \(e\), let \(x=e\cap f\), and
suppose \(d_H(z)\le D\) for every \(z\in f\setminus\{x\}\).  Let
\(\mathcal E_f\) be the set of edges meeting \(f\) away from \(x\) and
disjoint from \(e\).  For every point set \(X\) outside \(e\), the
number of incidences
\[
 \{(g,z):g\in\mathcal E_f,\ z\in g\cap X\}
\]
is at most
\[
 q|X|+(t-1)q^2.
\]
\end{lemma}

\begin{proof}
At most \(\Delta_i\) rows of color \(i\) are nonsingleton, so they
contribute at most \(q\Delta_i\le aq^2\) points.  By
Theorem~\ref{thm:least-deficient}(2), every singleton row has at most
\(\rho q\) points outside its packet, and color \(i\) has at most
\(tq\) rows.  This proves the first assertion.

If \(z\notin f\), at most one edge through \(z\) can correspond to
each of the \(q\) possible intersection points in \(f\setminus\{x\}\),
so its multiplicity in \(\mathcal E_f\) is at most \(q\).  If
\(z\in f\setminus\{x\}\), its multiplicity is at most
\(d_H(z)-1\le D-1\).  Since \(X\cap f\) has at most \(q\) points, the
incidence count is at most
\[
 q|X\setminus f|+(D-1)|X\cap f|
 \le q|X|+(D-1-q)q
 =q|X|+(t-1)q^2.
\]
\end{proof}

\begin{theorem}[Uncovered-pair criterion]
\label{thm:criterion}
Assume Theorem~\ref{thm:least-deficient}.  Choose
\[
 0<a<1,\qquad \frac tq<\beta<1-a,
\]
and \(b>0\) satisfying
\[
 b<G_{q,t},\qquad D^2b<q(q-t+1).
\]
If
\begin{equation}\label{eq:criterion}
 (t-1)-\frac{D^2b}{q^2}
 >
 \frac{a+t\rho+(t-1)/q}{1-a-\beta}
 +
 \frac{t^2\eta_*}{\beta(\beta-t/q)},
\end{equation}
then every retained neighbor \(f\) of \(e\) satisfies
\(\delta_H(f)\ge b\).
\end{theorem}

\begin{proof}
Suppose that a retained neighbor \(f\) satisfies \(\delta_H(f)<b\),
and put \(x=e\cap f=v_i\).  Since \(b<G_{q,t}\),
Proposition~\ref{prop:degree-stability} applied to \(f\) gives degree
at most \(D\) at every point of \(f\) and
\[
 L(f)\le D^2\delta_H(f)<D^2b<q(q-t+1).
\]
For each \(y\in f\setminus\{x\}\), at most \(q\) edges through \(y\),
other than \(f\), meet \(e\).  Hence
\[
 |\mathcal E_f|
 \ge\sum_{y\in f\setminus\{x\}}(d_H(y)-1-q)
 =(t-1)q^2-L(f)+\Delta_i
 >
 \left((t-1)-\frac{D^2b}{q^2}\right)q^2.
\]

Take \(g\in\mathcal E_f\).  Theorem~\ref{thm:capacity}, applied to
\(f\) with the prescribed edge \(g\), gives at least
\(q-\Delta_i\ge(1-a)q\) distinct intersections with color-\(i\) trace
rows at \(e\).  Every such intersection lies in
\[
 \Om:=\Om_1\cup\cdots\cup\Om_t
\]
or in \(X_i\).  Call \(g\) \emph{light} if
\(|g\cap\Om|<\beta q\).  A light edge contains more than
\((1-a-\beta)q\) points of \(X_i\).  Lemma~\ref{lem:spill} therefore
gives
\[
 |\mathcal E_f^{\mathrm{light}}|
 \le
 \frac{a+t\rho+(t-1)/q}{1-a-\beta}\,q^2.
\]
Every remaining edge contains at least \(\beta q\) packet points.
Theorem~\ref{thm:least-deficient}(5) bounds their number by
\[
 \frac{t^2\eta_*}{\beta(\beta-t/q)}\,q^2.
\]
Inequality \eqref{eq:criterion} contradicts the lower bound for
\(|\mathcal E_f|\).
\end{proof}

\section{Growing-crown repulsion}\label{sec:growing}

\begin{theorem}[Asymptotic local repulsion]\label{thm:growing}
Let \(q_j\to\infty\), let \(2\le t_j\le q_j\), and put
\[
 r_j=q_j+1,\qquad D_j=t_j q_j+1.
\]
Let \(e_j\) be an edge of a finite linear
\(C_{1,t_j+1}^{r_j}\)-free \(r_j\)-graph \(H_j\).  Assume
\[
 \frac{t_j^2}{q_j}\longrightarrow0,
 \qquad
 t_j^3\delta_{H_j}(e_j)\longrightarrow0.
\]
Then, for every fixed \(0<\theta<1\),
\[
 \frac{
 \bigl|\{f\ne e_j:f\cap e_j\ne\varnothing,\
 \delta_{H_j}(f)\ge\theta/t_j^2\}\bigr|
 }{r_j(D_j-1)}
 \longrightarrow1.
\]
\end{theorem}

\begin{proof}
Suppress the index \(j\).  Under \(t^2/q\to0\),
\[
 t^2G_{q,t}
 =
 \frac{t^2(q^2-(t-1)q-1)}
 {(tq+1)((t-1)q+1)}
 \ge1-o(1).
\]
The second hypothesis implies \(\delta_H(e)<G_{q,t}\) for all large
\(j\).  Proposition~\ref{prop:degree-stability} gives
\[
 \xi=\frac{L(e)}{q^2}
 \le\left(\frac Dq\right)^2\delta_H(e),
 \qquad
 t\xi\longrightarrow0.
\]
The condition of Proposition~\ref{prop:approx-blocking} holds
eventually, \(p_0=q+1-t\ge2\), and the packet construction is available.
By \eqref{eq:rho-bound} and \eqref{eq:eta-star-bound},
\[
 t\rho\longrightarrow0,\qquad
 \eta_*=\frac tq+O(\xi+\rho),\qquad
 t\eta_*\longrightarrow0.
\]

Set
\[
 \beta=\frac12,\qquad b=\frac{\theta}{t^2},
\]
and choose a fixed
\[
 0<a<\min\left\{\frac18,\frac{1-\theta}{16}\right\}.
\]
The two threshold conditions in Theorem~\ref{thm:criterion} hold for
all large \(j\): one has \(b<G_{q,t}\), while
\[
 \frac{D^2b}{q^2}
 =\theta\left(1+\frac1{tq}\right)^2<1-\frac{t-1}{q}.
\]
For \eqref{eq:criterion}, the left side is
\(t-1-\theta-o(1)\).  The first term on the right is
\(a/(1/2-a)+o(1)\).  The second term is \(o(t)\), because
\(t\eta_*\to0\).  If \(t\) is bounded, it is \(o(1)\); if
\(t\to\infty\), divide the inequality by \(t\).  The choice of \(a\)
then makes \eqref{eq:criterion} valid for all large \(j\).
Thus every retained neighbor has defect at least \(\theta/t^2\).

It remains to count retained neighbors.  The \(t\) colors used by the
blocking matching contain at most \(t^2q\) rows.  The bad unused colors
contain at most \((t/a)L(e)\) rows, and the nonsingleton rows in the
good unused colors contribute at most \(L(e)\).  Hence at most
\[
 t^2q+\left(\frac ta+1\right)L(e)
\]
neighboring edges are discarded.  The actual number of neighbors of
\(e\) is
\[
 \sum_{v\in e}(d_H(v)-1)
 =r(D-1)-L(e).
\]
Dividing the total shortage by
\(r(D-1)=tq(q+1)\) gives
\[
 O\left(\frac tq+\frac{\xi}{a}+\frac{\xi}{t}\right)=o(1).
\]
The claimed limit follows.
\end{proof}

\begin{corollary}[Uniform local repulsion]\label{cor:uniform-local}
There are absolute constants \(Q_0<\infty\) and \(\varepsilon_0>0\)
such that the following holds.  If
\[
 q\ge Q_0t^2
\]
and \(e\) is an edge of a finite linear
\(C_{1,t+1}^{q+1}\)-free hypergraph with
\[
 \delta_H(e)<\frac{\varepsilon_0}{t^3},
\]
then at least
\[
 \frac12r(D-1)
\]
neighbors \(f\) of \(e\) satisfy
\[
 \delta_H(f)\ge\frac1{100t^2}.
\]
\end{corollary}

\begin{proof}
Set
\[
 a=\frac1{100},\qquad \beta=\frac12,
 \qquad b=\frac1{100t^2}.
\]
Choose \(\varepsilon_0>0\) so that
\(\varepsilon_0\le1/100\) and \(132\varepsilon_0<1/100\).  Then choose
\(Q_0\) so large that all inequalities used below, including the final
retention estimate, hold whenever \(q\ge Q_0t^2\).  Assume
\(\delta_H(e)<\varepsilon_0/t^3\).  Then
\(\delta_H(e)<G_{q,t}\), so Proposition~\ref{prop:degree-stability}
gives
\[
 \xi\le\left(\frac Dq\right)^2\frac{\varepsilon_0}{t^3}
 \le\frac{2\varepsilon_0}{t}.
\]
Equation \eqref{eq:rho-bound} also gives
\[
 \rho\le\frac{9\varepsilon_0}{t},
\]
and therefore
\[
 t\eta_*
 =\frac{t^2}{q}+2t\xi+2t\rho
 \le Q_0^{-1}+22\varepsilon_0.
\]
The same choice of \(Q_0\) ensures
\(b<G_{q,t}\) and \(D^2b<q(q-t+1)\).

It remains to verify \eqref{eq:criterion}.  Its first term on the
right is at most
\[
 \frac{1/100+9\varepsilon_0+1/(2Q_0)}{49/100}.
\]
Moreover \(t/q\le1/(2Q_0)\), and hence its second term is at most
\[
 6t\bigl(Q_0^{-1}+22\varepsilon_0\bigr).
\]
The left side is at least
\[
 t-1-\frac1{100}\left(1+\frac1{tq}\right)^2.
\]
The choice of \(\varepsilon_0\) makes the terms independent of \(Q_0\)
sufficiently small, and the choice of \(Q_0\) makes the remaining
\(Q_0^{-1}\)-terms small enough that the right side is below the displayed
left side for every \(t\ge2\).  Hence Theorem~\ref{thm:criterion} applies
uniformly.

The number of missing or nonretained neighbors is at most
\[
 L(e)+t^2q+\left(\frac ta+1\right)L(e).
\]
Dividing by \(r(D-1)=tq(q+1)\) gives at most
\[
 \frac{t}{q+1}
 +\left(\frac1a+\frac2t\right)\xi.
\]
By the same choices of \(\varepsilon_0\) and \(Q_0\), this quantity is less than \(1/2\).  Thus at least half of the neighbors are
retained, and Theorem~\ref{thm:criterion} gives
\(\delta_H(f)\ge1/(100t^2)\) on all of them.
\end{proof}

\section{A uniform Tur\'an gap}\label{sec:global}

For a finite hypergraph \(H\), put
\[
 \Def(H)=\sum_{e\in E(H)}\delta_H(e).
\]
Double counting reciprocal incidences gives
\begin{equation}\label{eq:def-identity}
 |V^+(H)|=\frac rD|E(H)|+\Def(H),
\end{equation}
where \(V^+(H)\) is the set of nonisolated vertices.

\begin{corollary}[Growing-crown Tur\'an gap]\label{cor:global-gap}
There are absolute constants \(Q_0<\infty\) and \(c_0>0\) such that
\[
 q\ge Q_0t^2
 \quad\Longrightarrow\quad
 \clin_{q+1,t+1}
 \le t-\frac{c_0}{t},
\]
where
\[
 \clin_{r,k}
 =
 \limsup_{n\to\infty}
 \frac{\operatorname{ex}^{\mathrm{lin}}_r(n,C^r_{1,k})}{n}.
\]
\end{corollary}

\begin{proof}
Take \(Q_0,\varepsilon_0\) from
Corollary~\ref{cor:uniform-local}, and put
\[
 A=\frac{\varepsilon_0}{t^3},
 \qquad
 \lambda=\frac12,
 \qquad
 b=\frac1{100t^2}.
\]
Let
\[
 \mathcal G=\{e\in E(H):\delta_H(e)<A\},
 \qquad
 m=|E(H)|,\qquad x=|\mathcal G|.
\]
Edges outside \(\mathcal G\) give
\[
 \Def(H)\ge A(m-x).
\]
Every \(e\in\mathcal G\) has at least
\(\lambda r(D-1)\) neighboring witness edges of defect at least \(b\).
A fixed witness \(f\) is counted for at most \(r(D-1)\) members of
\(\mathcal G\).  Indeed, if \(e\in\mathcal G\) meets \(f\) at \(v\),
degree stability at \(e\) gives \(d_H(v)\le D\); for each of the \(r\)
points of \(f\) there are at most \(D-1\) other edges through that
point.  Linearity counts an edge at most once, because two distinct points of \(f\) cannot lie on the same other edge.
There are therefore at least \(\lambda x\) distinct witnesses, and
\[
 \Def(H)\ge\lambda bx.
\]
Consequently
\[
 \Def(H)
 \ge
 \max\{A(m-x),\lambda bx\}
 \ge
 \frac{A\lambda b}{A+\lambda b}\,m.
\]
Because \(\varepsilon_0\le1/100\) and \(t\ge2\),
\[
 \frac{A\lambda b}{A+\lambda b}
 \ge\frac{\varepsilon_0}{2t^3}.
\]
Equation \eqref{eq:def-identity} now gives
\[
 \frac{|E(H)|}{|V(H)|}
 \le
 \frac{1}{r/D+\varepsilon_0/(2t^3)}.
\]
Write \(y=D/r\).  Under \(q\ge Q_0t^2\), one has
\(t/2\le y\le t\).  Hence
\[
 \frac{1}{r/D+\varepsilon_0/(2t^3)}
 =
 \frac{y}{1+y\varepsilon_0/(2t^3)}
 \le
 \frac{t}{1+\varepsilon_0/(4t^2)}
 \le
 t-\frac{\varepsilon_0}{8t},
\]
using \(\varepsilon_0\le1\).  The result follows with
\(c_0=\varepsilon_0/8\).
\end{proof}

\section{Conclusion}\label{sec:conclusion}

Theorem~\ref{thm:growing} shows that the equality configuration is
locally repelling when \(t^2/q\to0\): if one edge has
\(t^3\delta_H(e)=o(1)\), then almost all of its neighbors have defect at
least a fixed multiple of \(t^{-2}\).  Corollary~\ref{cor:global-gap}
turns the uniform form of this statement into a deficit of order \(1/t\)
from the coefficient \(t\).

The two error scales in the local theorem have distinct sources.  The term
\(t^2/q\) measures the loss caused by the \(t\) colors used in a blocking
matching, while \(t^3\delta_H(e)\) controls the accumulated degree
shortfall on the base edge.  Extending the argument beyond
\(t^2/q=o(1)\) would require a sharper use of the pair coverage contributed
by the colors in the blocking matching.

\end{document}